\documentclass{amsart}
\usepackage{amssymb}
\usepackage{latexsym}
\usepackage{euscript}

\def\cH{{\mathcal H}}
\def\cO{{\mathcal O}}
\def\cD{{\mathcal D}}
\def\dD{{\mathbb D}}

 \def\sN{{\mathfrak N}}

\newtheorem{theorem}{Theorem}[section]
\newtheorem{proposition}[theorem]{Proposition}
\newtheorem{corollary}[theorem]{Corollary}
\newtheorem{lemma}[theorem]{Lemma}

\theoremstyle{definition}

\newtheorem{remark}[theorem]{Remark}
\newtheorem{definition}[theorem]{Definition}

\numberwithin{equation}{section}

\newcommand {\sk}[3]{\left#1#2\right#3}  
\newcommand {\e}{\varepsilon}
\newcommand {\h}{\hat}
\newcommand {\wh}{\widehat}
\renewcommand {\l}{\lambda}
\renewcommand {\k}{\kappa}
\newcommand {\s}{\sigma}
\renewcommand {\t}{\theta}
\newcommand {\T}{\Theta}
\newcommand {\g}{\gamma}
\renewcommand {\L}{\Lambda}

\newcommand {\ov}{\overline}
\newcommand {\m}{\mu}
\newcommand {\G}{\Gamma}
\renewcommand {\r}{\rho}

\newcommand {\w}{\widetilde}
\newcommand {\p}{\perp}

\def\spn{{\rm span\,}}
\def\ran{{\rm ran\,}}

\def\ind{{\rm ind\,}}
\def\dom{{\rm dom\,}}
\def\ker{{\rm ker\,}}

\def\mul{{\rm mul\,}}
\newcommand {\zx}{{[*]}}
\newcommand {\sx}{{-[*]}}

\begin{document}

\title[Generalized resolvents of isometric operators]
{Generalized resolvents of isometric operators  in Pontryagin
spaces}
\author{Dmytro Baidiuk}
\date{30.12.2015}


\begin{abstract}
An isometric operator $V$ in a Pontryagin space $\cH$ is called
standard, if its domain and the range are nondegenerate subspaces in
$\cH$. Generalized resolvents of standard isometric operators were
described in \cite{DLS90}. In the present paper generalized
resolvents of non-standard Pontryagin space isometric operators are
described. The method of the proof is based on the notion of
boundary triplet of isometric operators in Pontryagin spaces. In the
Hilbert space setting the notion of boundary triplet for isometric
operators was introduced in \cite{MM03}.
\end{abstract}
\maketitle

\section{Introduction} Unitary operators in Pontryagin spaces and
the problem of continuation of an isometric operators $V$ were
studied in papers  \cite{AI86}, \cite{IK65}, \cite{EI82},
\cite{KKhYu87}, \cite{L71}. An isometric operator $V$ in a
Pontryagin space $\cH$ is called standard, if its domain and the
range are nondegenerate subspaces in $\cH$. A description of
generalized resolvents of a standard isometric operator was given in
\cite{DLS90}. For a nonstandard isometric operator, the approach
presented in \cite{N001}, \cite{N002} leads to significant technical
problems related to the necessity to consider unitary linear
relations in a Pontryagin space.

 We propose another approach to the theory of extensions of isometric
 operators in Pontryagin spaces that is based on the notion
 of the boundary triplet of an isometric operator. In the case of a Hilbert space $\cH$,
 this notion was introduced and applied to the classical problems of analysis in
 works by M. M. Malamud and V. I. Mogilevskii \cite{MM03} and \cite{MM04}.
 For a Pontryagin space setting the definition of boundary triplet of isometric
 operator given in \cite{B09} is a partial case of the definition of boundary relation in \cite{DHMS04}.
Boundary triplets considered in \cite{B13} had the property that the
auxiliary spaces of that triplets were Hilbert spaces. Such objects
were sufficient in order to give a description of generalized
resolvents of $V$ corresponding to unitary extensions of $V$ acting
in wider spaces $\w\cH$ with the same negative index as $\cH$.

In the present paper we give a description of generalized resolvents
of an isometric operator $V$ which correspond to exit space unitary
extensions of $V$ acting in spaces $\w\cH$ with negative index
$\w\k$ exceeding the negative index of $\cH$. It turned out that the
notion of boundary  triplet with Hilbert auxiliary space introduced
in \cite{B13} is not sufficient for this purpose. That is why we
extend the notion of boundary triplet for Pontryagin space isometric
operator to the case where the auxiliary space is a Pontryagin
space.

 We introduce the notion of the Weyl function of an isometric operator, which generalizes the
 appropriate definition from \cite{MM03} and study its properties.
 This allows us to describe the properties of extensions of
 the operator $V$, as well as generalized resolvents of the
 isometric operator $V$ in a Pontryagin space.

The author is grateful to his scientific supervisor, V. A. Derkach,
for numerous discussions and useful remarks and to M. M. Malamud and
V. I. Mogilevskii for the possibility to read the manuscript
containing the proofs of all propositions in \cite{MM03}.

\section{Preliminary information}
\begin{subsection}{Linear relations}
We recall some information about linear relations from \cite{Ben72},
\cite{DHMS04}. Let $\cH_1$ and $\cH_2$ be Hilbert spaces. A linear
relation (l.r.) $T$ from $\cH_1$ to $\cH_2$ is a linear subspace
  in $\cH_1 \times \cH_2$. If the linear operator $T$
is identified with its graph, then the set
$\mathcal{B}(\cH_1,\cH_2)$ of linear bounded operators from $\cH_1$
to $\cH_2$ is contained in the set of linear relations from $\cH_1$
to $\cH_2$. In what follows, we interpret the linear relation
$T:\cH_1\to\cH_2$ as a multivalued linear mapping from $\cH_1$ to
$\cH_2$. If $\cH:=\cH_1=\cH_2$ we say that $T$ is a linear relation
in $\cH$.

For the linear relation $T:\cH_1\to \cH_2$, we denote by $\dom T$,
$\ker T$, $\ran T$, and $\mul T$ the domain, the kernel, the range,
and the multivalued part of $T$, respectively. The inverse relation
$T^{-1}$ is a linear relation from $\cH_2$ to $\cH_1$ defined by the
equality
$${T^{-1}=\sk\{{\begin{bmatrix} f' \\ f
\end{bmatrix}:\begin{bmatrix} f \\ f' \end{bmatrix}\in T}\}}.$$
 The sum $T+S$
of two linear relations $T$ and $S$ is defined by
\begin{equation}
  T+S=\sk\{{\begin{bmatrix} f \\ g+h \end{bmatrix}:\begin{bmatrix} f \\ g \end{bmatrix}\in T,
  \begin{bmatrix} f \\ h \end{bmatrix}\in S}\}.
\end{equation}

 Let $\cH_1$ and $\cH_2$ be Banach spaces. By $\mathcal{B}(\cH_1,\cH_2)$, we denote the set of all linear bounded
 operators from ${\cH_1}$ to $\cH_2$;
$\mathcal{B}({\cH}):=\mathcal{B}(\cH,\cH)$. We recall that the point
$\l\in\mathbb{C}$ is called a point of regular type of an operator
$T\in\mathcal{B}({\cH})$, if there exists $c_{\l}>0$ such that
\begin{equation}\label{E:rtp}
  \|(T-\l I)f\|_{\cH}\geq c_{\l}\|f\|_{\cH},\quad f\in\cH.
\end{equation}

If $\ran(T-\l I)=\cH$ and \eqref{E:rtp} holds, then $\l$ is called a
regular point of the operator $T$. Let $\rho(T)$ ($\wh\rho(T)$) be
the set of regular (regular type) points of the operator $T$ and let
\begin{equation}\label{res}
R_\l(T):=(T-\l I)^{-1}, \quad\l\in\r(T).
\end{equation}

\end{subsection}

\begin{subsection}{Linear relations in Pontryagin spaces}
 Let $\mathcal{H}$ be a Hilbert space, and let $j_{\mathcal{H}}$ be
 a signature operator in $\cH$, i.e.,
$j_{\mathcal{H}}=j_{\mathcal{H}}^*=j_{\mathcal{H}}^{-1}$. We
interpret the space $\mathcal{H}$ as a Kre\u{\i}n space
$(\mathcal{H},j_{\mathcal{H}})$ (see \cite{AI86}), in which the
indefinite scalar product is defined by the equality $ [\varphi,
\psi]_{\mathcal{H}}=
 (j_{\mathcal{H}} \varphi, \psi)_{\mathcal{H}}.
$ The signature operator $j_\mathcal{H}$ can be presented in the
form $j_\mathcal{H}=P_+-P_-$, where $P_+$ and $P_-$ orthoprojectors
in $\mathcal{H}$. In the case where $P_-$ is finite-dimensional, and
$\mbox{dim}P_-\mathcal{H}=\kappa$, the Kre\u{\i}n space
$(\mathcal{H},j_\mathcal{H})$ is called a Pontryagin space with
negative index $\kappa$, which is denoted by
$\mbox{ind}_-\mathcal{H}=\kappa$.

 Consider two Pontryagin spaces $(\mathcal{H}_1,j_{\mathcal{H}_1})$ and $(\mathcal{H}_2,j_{\mathcal{H}_2})$ and
 a linear relation $T$ from $\mathcal{H}_1$ to $\mathcal{H}_2$. Then the adjoint linear relation
 $T^{[*]}$ consists of pairs $\begin{bmatrix} g_2 \\g_1
\end{bmatrix}\in\cH_2\times\cH_1$ such that
$$
[f_2,g_2]_{\cH_2}=[f_1,g_1]_{\cH_1},\text{ for all }\begin{bmatrix}
f_1
\\f_2
\end{bmatrix}\in T.
$$
If $T^*$ is the l.r. adjoint to $T$ considered as a l.r. from the
Hilbert space $\cH_1$ to the Hilbert space $\cH_2$, then
$T^{[*]}=j_{\cH_1}T^*j_{\cH_2}$.

The l.r. $T^{[*]}$ satisfies the equalities
\begin{equation}
\label{adjequ}
 (\dom T)^{[\perp]}=\mul T^{[*]}, \quad
 (\ran T)^{[\perp]}=\ker T^{[*]},
\end{equation}
 where the sign $[\perp]$ means the orthogonality in a Pontryagin space.
\begin{definition}\label{D:2.1}
 A linear relation $T$ from a Pontryagin space $(\cH_1,j_{\cH_1})$
 to a Pontryagin space
  $(\cH_2,j_{\cH_2})$ is called isometric, if, for all $\begin{bmatrix} \varphi \\\varphi'
\end{bmatrix}\in
  T$, the equality
\begin{equation}
\label{iso}
 [ \varphi', \varphi']_{\cH_2} = [ \varphi, \varphi]_{\cH_1}
\end{equation}
holds. It is called a contractive one, if the equality
 \eqref{iso} is replaced by an inequality with the sign $\le$.
A linear relation $T$ from $(\cH_1,j_{\cH_1})$ to
$(\cH_2,j_{\cH_2})$ is called unitary, if $T^{-1} = T^{[*]}$. It
follows from \eqref{iso} that a linear relation $T$ is isometric iff
$T^{-1} \subset T^{[*]}$.
\end{definition}
As is known \cite{AI86}, the sets $\dD\setminus\wh\r(T)$ and
$\dD_e\setminus\wh\r(T)$ for an isometric operator $T$ in a
Pontryagin space with $\ind_-\cH=\k$ consist of at most $\k$ points,
which belong to $\s_p(T)$.

 The definition of unitary relation was first given in
\cite{Sh76}, where the following assertion was proved.
\begin{proposition}\label{Plo}
If $T$ is a unitary relation, then
\begin{enumerate}
\item
$\dom T$ is closed iff $\ran T$ is closed;
\item
the equalities  $ \ker T=\dom T^{[\p]}$, $\mul T=\ran T^{[\p]} $
hold.
\end{enumerate}
\end{proposition}
Proposition \ref{Plo} yields the following result.
\begin{corollary}
If $T$ is a unitary relation in a Pontryagin space, then $\mul
T\neq\{0\}$ if and only if $\ker T\neq\{0\}$. In this case,
$\dim\mul T=\dim\ker T$.
\end{corollary}

\end{subsection}

\section{Boundary triplets for an isometric operator in a Pontryagin space}
\begin{subsection}{Boundary triplets and description of extensions of an isometric operator in a Pontryagin space }
In the case where $\cH$ is a Hilbert space, the definition of the
boundary triplet for an isometric operator was introduced in
\cite{MM03}. We note that the notion of the boundary triplet of an
isometric operator, which will be introduced below in Definition
\ref{G3}, is a partial case of the notion of the boundary relation
of an isometric operator in a Pontryagin space \cite{B09}. The
notion of boundary triplets for symmetric operator was introduced in
\cite{Cal39} (see also \cite{GG} and references therein).

\begin{definition}\label{G3}
Let $\mathcal{H}, \sN_1$ and $\sN_2$ be Pontryagin spaces with
negative indices $\kappa$ and $\ind_-\sN_1=\ind_-\sN_2=\kappa_1$
respectively. Let an operator $V:\mathcal{H}\rightarrow\mathcal{H}$
be an isometry in $\cH$. The collection
$\Pi=\{\sN_1\oplus\sN_2,\Gamma_1,\Gamma_2\}$ is called the boundary
triplet of the isometric operator $V$, if
\begin{itemize}
\item[1)]
the following Green's generalized identity holds:
\begin{equation}\label{TG}
[f',g']_\mathcal{H}-[f,g]_\mathcal{H}=[\Gamma_1 \wh{f},\Gamma_1
\wh{g}]_{\mathfrak{N}_1}-[\Gamma_2 \wh{f},\Gamma_2
\wh{g}]_{\mathfrak{N}_2},
\end{equation}
where $\wh{f}=\begin{bmatrix} f
\\f'
\end{bmatrix}$,
$\wh{g}=\begin{bmatrix} g \\g'
\end{bmatrix}\in V^{-[*]}$;
\item[2)]
the mapping
$\Gamma=(\Gamma_1,\Gamma_2)^T:V^{-[*]}\rightarrow\mathfrak{N}_1\oplus\mathfrak{N}_2$
 is surjective.
\end{itemize}
\end{definition}

For an isometric operator, it is convenient to define the defect
subspace $\sN_\l(V)$ as follows:
\begin{equation}\label{def}
\sN_\l(V):=\ker\sk({I-\l V^\zx})=\sk\{{f_\l:\begin{bmatrix} f_\l\\\l
f_\l
\end{bmatrix}\in
V^\sx}\}.
\end{equation}
For $\l\in\wh\r(V)$, $\sN_\l(V)$ is a closed subspace in $\cH$
\cite{AI86}.

 We also set
\begin{equation}\label{def2}
\wh{\sN}_\l(V):=\sk\{{\begin{bmatrix} f_\l\\\l f_\l
\end{bmatrix}:f_\l\in\sN_\l(V)}\}.
\end{equation}
It follows from (\ref{def}) that $\wh{\sN}_\l(V)\subset V^\sx$.

Let $\t$ be a linear relation from $\sN_2$ to $\sN_1$. We define the
extension $V_\t$ of the operator $V$ by the equality
\begin{equation}\label{E:1.2}
 V_\t=\sk\{{\wh f\in V^\sx:\begin{bmatrix} \G_2\wh{f} \\\G_1\wh{f}
\end{bmatrix}\in\t}\}.
\end{equation}
The extension $V_\t$ is, generally speaking, a linear relation from
$\cH$ to $\cH$.

We define two extension $V_1$ and $V_2$ of the operator $V$:
\begin{equation}\label{E:1.3}
 V_i=\sk\{{\wh f\in V^\sx:\G_i\wh{f}=0}\},\quad i=1,2.
\end{equation}
We note also that
\begin{equation}\label{E:1.23}
 V=\sk\{{\wh f\in V^\sx:\G_1\wh{f}=0\text{ and }\G_2\wh{f}=0}\}.
\end{equation}

(A) Here and further on the sets
\begin{equation}\label{Lambda1}
\L_1=\{\l\in\dD_e:\wh \sN_\l(V)\cap V_1\neq\{0\}\}=\s(V_1)\cap\dD_e;
\end{equation}
\begin{equation}\label{Lambda2}
\L_2=\{\l\in\dD:\wh \sN_\l(V)\cap V_2\neq\{0\}\}=\s(V_2)\cap\dD.
\end{equation}
 will be supposed to
be discrete, isolated. Denote $\cD_1:=\dD_e\setminus\L_1$ and
$\cD_2:=\dD\setminus\L_2$ the subsets of regular points of these
extensions.

\begin{lemma}\label{L:1.2}
Let the collection $\Pi=\sk\{{\sN_1\oplus\sN_2,\G_1,\G_2}\}$ be a
boundary triplet of the isometric operator $V$. Then:
\begin{enumerate}
\item[(1)]
For all $\l\in\cD_1=\r(V_1)\cap \dD_e$
\begin{equation}\label{E:1.4}
V^\sx=V_1\dotplus\wh\sN_\l(V);
\end{equation}
\item[(2)]
For all $\l\in\cD_2=\r(V_2)\cap \dD$
\begin{equation}\label{E:1.5}
V^\sx=V_2\dotplus\wh\sN_\l(V).
\end{equation}

\end{enumerate}
\end{lemma}
\begin{proof}
We prove equality (\ref{E:1.4}) (equality (\ref{E:1.5}) can be
proved analogously). For this purpose, we set the inclusion
$V^\sx\subset V_1\dotplus\wh \sN_\l(V)$. Consider a pair of vectors
$\begin{bmatrix} f
\\f'
\end{bmatrix}\in V^\sx$.
 Let $f_1=(V_1-\l)^{-1}(f'-\l f)$ be a solution of the equation
$$
f'-\l f=f_1'-\l f_1,\text{ where }\begin{bmatrix} f_1
\\f'_1
\end{bmatrix}\in V_1,
$$
which is determined uniquely for $\l\in\cD_1$. Then
$f'-f'_1=\l(f-f_1)$, i.e., $\begin{bmatrix} f-f_1
\\\l(f-f_1)
\end{bmatrix}\in V^\sx$ and, hence, $f-f_1\in\sN_\l(V)$. Since the inverse inclusion is obvious,
equality (\ref{E:1.4}) is proved.
\end{proof}
The following theorem gives a description of proper extensions of
the operator $V$, i.e., such that $V\subset V_\t\subset V^\sx$.
\begin{theorem}\label{T:1}
Let the collection $\Pi=\sk\{{\sN_1\oplus\sN_2,\G_1,\G_2}\}$ be a
boundary triplet for $V,$ let $\t$ be a linear relation from $\sN_2$
to $\sN_1$, and let $V_\t$ be the corresponding extension of the
operator $V$. Then:
\begin{enumerate}
\item[(1)]
the inclusion $V_{\t_1}\subset V_{\t_2}$ is equivalent to the
inclusion $\t_1\subset\t_2$;
\item[(2)]
$V_{\t^{-[*]}}=V_\t^\sx$;
\item[(3)]
$V_\t$ is a unitary extension of the operator $V$, iff $\t$ is the
graph of a unitary l.r. from $\sN_2$ to $\sN_1$;
\item[(4)]
$V_\t$ is an isometric extension of the operator $V$, iff $\t$ is
the graph of an isometric l.r. from $\sN_2$ to $\sN_1$;
\item[(5)]
$V_\t$ is a coisometric extension of the operator $V$, iff $\t$ is
the graph of a coisometric l.r. from $\sN_2$ to $\sN_1$;
\item[(6)]
$V_\t$ is a contraction, iff $\t$ is a contraction;
\item[(7)]
$V_\t$ is an expansion, iff $\t$ is an expansion.
\end{enumerate}
\end{theorem}
\begin{proof}
Assertion 1) follows obviously from the definition of $V_{\t_1}$ and
$V_{\t_2}$.

2) We take $\begin{bmatrix} f \\f'
\end{bmatrix}\in V_\t$ and $\begin{bmatrix} g \\g'
\end{bmatrix}\in V_\t^\sx$.
 Then $\begin{bmatrix} g' \\g
\end{bmatrix}\in V_\t^\zx$. From ~(\ref{TG}), we obtain
$$
0=[f',g']-[f,g]=\sk[{\G_1\wh{f},\G_1\wh{g}}]_{\sN_1}-\sk[{\G_2\wh{f},\G_2\wh{g}}]_{\sN_2}
$$
Since $\begin{bmatrix}\G_2\wh f \\\G_1\wh f
\end{bmatrix}\in\t$, we have $\begin{bmatrix}\G_1\wh g \\\G_2\wh g
\end{bmatrix}\in\t^{[*]}$ or
$\begin{bmatrix}\G_2\wh g \\\G_1\wh g
\end{bmatrix}\in\t^{-[*]}$, which means
$\wh{g}\in V_{\t^{-[*]}}$. Hence, we show that $V_\t^\sx\subset
V_{\t^{-[*]}}$.

The inverse assertion can be proved by inversion of the above
reasoning.

3) Let $V_\t^\sx=V_\t$, i.e., let $V_\t$ be a unitary extension of
the operator $V$. Using the first assertion of this Theorem, we
obtain $\t^{-[*]}=\t$. Conversely, we set $\t^{-[*]}=\t$ and, by the
first assertion of this Theorem, arrive at $V_\t^\sx=V_\t$.

4) and 5) are proved analogously. Assume that $V_\t$ is a
coisometry, i.e., $V_\t^{-1}\supset V_\t^\zx$.Then, by virtue of
item 2), $V_{\t^{-[*]}}=V_\t^\sx\subset V_\t$. By virtue of assertion
1), we obtain $\t^{-[*]}\subset\t$, i.e., $\t$ is a coisometry.

6) Let $V_\t$ be a contraction. Then, for $\wh f=\begin{bmatrix}
f\\f'
\end{bmatrix}\in V_\t$, formula (\ref{TG}) yields
$$
0\geq[f',f']-[f,f]=\sk[{\G_1\wh{f},\G_1\wh{f}}]_{\sN_1}-\sk[{\G_2\wh{f},\G_2\wh{f}}]_{\sN_2}.
$$
We obtain
$\sk[{\G_1\wh{f},\G_1\wh{g}}]_{\sN_1}\leq\sk[{\G_2\wh{f},\G_2\wh{g}}]_{\sN_2}$.
This means that $\t$ is a contraction.

7) is proved analogously 6).
\end{proof}
\begin{corollary}\label{C:V12}
  The extensions $V_1$ and $V_2$ of the operator $V$ are connected by the formula
  \begin{equation}\label{E:V12}
    V_2=V_1^{-[*]}.
  \end{equation}
\end{corollary}
\begin{proof}
  It is gotten by using the second assertion of the previous Theorem with $\t=0$. Because in this case
  $V_\t=V_1$ and $V_{\t^{-[*]}}=V_2$.
\end{proof}
\end{subsection}
\begin{subsection}{$\g$-field and Weyl function.}
The notion of the Weyl function of an isometric operator $V$ in a
Hilbert space, which allows one to describe the analytic properties
of extensions of the operator $V$, was introduced in \cite{MM03}. In
this section, we will generalize this notion to the case of
isometric operator $V$  in a Pontryagin space.
\begin{lemma}\label{L:1.4}
Let $\Pi=\sk\{{\sN_1\oplus\sN_2,\G_1,\G_2}\}$ be the boundary
triplet for $V$, and let $V_1$ and $V_2$ be the extensions of
the isometric operator $V$ that are defined in (\ref{E:1.3}).
Then the mappings
 $\G_j\upharpoonright\wh{\sN}_\l(V):\wh{\sN}_\l(V)\to\sN_j\quad
j=1,2$, are bounded and boundedly invertible for
$\l\in\cD_j$.

In this case, the operator-functions
\begin{equation}\label{E:1.7}
\g_j(\l):=\pi_1\wh{\g}_j(\l)=\pi_1\sk({\G_j\upharpoonright\wh{\sN}_\l(V)})^{-1}
\end{equation}
satisfy the equality
\begin{equation}\label{E:1.8}
\g_j(\l)=(I+(\l-\m)(V_j-\l)^{-1})\g_j(\m),\quad\text{ for }
\l,\m\in\cD_j,\text{ and }j=1,2.
\end{equation}
The operator-functions $\g_j(\cdot)$ are called $\g$-fields for the l.r.
$V^\sx$.
\end{lemma}
\begin{proof}
First, we will show that the mapping $\G:V^\sx\to\begin{bmatrix}
\sN_1\\\sN_2
\end{bmatrix}$ is closed. Let $\wh f=\begin{bmatrix}  f_n\\f'_n
\end{bmatrix}\in V^\sx$ and $\wh f_n\to 0$ . Then $\G\wh f_n=\begin{bmatrix} \G_1 \wh f_n \\\G_2 \wh f_n
\end{bmatrix}\to \begin{bmatrix}  h_1\\h_2
\end{bmatrix}=:h$. From ~(\ref{TG}), we obtain
$[h_1,\G_1\wh{g}]_{\sN_1}-[h_2,\G_2\wh{g}]_{\sN_2}=0$. The
surjectivity of $\G$ and nondegeneracy of $\sN_1$ and $\sN_2$, imply
that  $h=0$.

Since $\dom \G=V^\sx$, the operator $\G$ is bounded by the closed
graph Banach theorem. Hence, $\G_1$ and $\G_2$ are bounded as well.

By virtue of equality (\ref{E:1.4}) and the surjectivity of $\G$,
the mapping
$${\G_1\upharpoonright\wh{\sN}_\l(V):\wh{\sN}_\l(V)\to\sN_1}$$ acts on
all $\sN_1$.
Hence, $\G_1\upharpoonright\wh{\sN}_\l(V)$ is boundedly
invertible.

Analogously, we can prove the bounded invertibility of
$\G_2\upharpoonright\wh{\sN}_\l(V)$. Hence, $\g_j(\l)$ for
$\l\in\cD_j$ $j=1,2$ are defined properly.

We now prove identity (\ref{E:1.8}). For definiteness, we take
$j=1$ and will prove that
$$
\g_1(\l)=(I+(\l-\m)(V_1-\l)^{-1})\g_1(\m),\text{ for }
\l,\m\in\cD_1.
$$
Consider the vector $g_\m=\g_1(\m)h_1\in\sN_\m(V)$, where $h_1\in\sN_1$.
 Then there exists the vector $h_2\in\sN_2$ such that $\G\begin{bmatrix}
g_\m\\\m g_\m
\end{bmatrix}=\begin{bmatrix} h_1 \\h_2
\end{bmatrix}.$
We set
$$
f_\l=g_\m+(\l-\m)(V_1-\l)^{-1}g_\m.
$$
Then
\begin{equation}\label{01}
\wh f_\l=\begin{bmatrix} g_\m \\\m g_\m
\end{bmatrix}+(\l-\m)\begin{bmatrix} (V_1-\l)^{-1}g_\m
\\\sk({I+\l(V_1-\l)^{-1}})g_\m
\end{bmatrix}.
\end{equation}
In this equality,
\begin{multline*}
\wh g_\m=\begin{bmatrix} g_\m \\\m g_\m
\end{bmatrix}\in\wh\sN_\m(V)\subset V^\sx,
\begin{bmatrix} (V_1-\l)^{-1}
\\I+\l(V_1-\l)^{-1}
\end{bmatrix}g_\m\in V_1\subset V^\sx.
\end{multline*}
 Thus, $\wh f_\l\in\wh\sN_\l(V)$.

 Below, we will use an equality that follows from (\ref{E:1.3})
\begin{equation}\label{02}
\G_1\begin{bmatrix} (V_1-\l)^{-1}g_\m \\\sk({I+\l(V_1-\l)^{-1}})g_\m
\end{bmatrix}=0
\end{equation}
Equalities (\ref{01}) and (\ref{02}) yield
\begin{equation*}
\G_1 \wh f_\l=\G_1\begin{bmatrix} g_\m \\\m g_\m
\end{bmatrix}+(\l-\m)\G_1\begin{bmatrix} (V_1-\l)^{-1}g_\m
\\\sk({I+\l(V_1-\l)^{-1}})g_\m
\end{bmatrix}= h_1
\end{equation*}
Hence, $f_\l=\pi_1\G_1^{-1}h_1=\g_1(\l)h_1$. This proves
(\ref{E:1.8}).
\end{proof}
The previous lemma implies that it is possible to define the
operator-functions $M_1(\cdot)$ and $M_2(\cdot)$:
\begin{equation}\label{D:M1}
M_1(\l)\G_1\upharpoonright\sN_\l(V)=\G_2\upharpoonright\sN_\l(V),\quad\l\in\cD_1;
\end{equation}
\begin{equation}\label{D:M2}
M_2(\l)\G_2\upharpoonright\sN_\l(V)=\G_1\upharpoonright\sN_\l(V),\quad\l\in\cD_2.
\end{equation}
It follows from definition (\ref{E:1.7}) of $\g_1(\cdot)$ and
$\g_2(\cdot)$ that $M_1(\l)$ and $M_2(\l)$ are defined properly, and
\begin{equation}\label{MG1}
M_1(\l):=\G_2\wh{\g}_1(\l),\quad\l\in\cD_1;
\end{equation}
\begin{equation}\label{MG2}
M_2(\l):=\G_1\wh{\g}_2(\l),\quad\l\in\cD_2.
\end{equation}
In what follows, we need the Schur class $S$ and the generalized
Schur class  $S_\k$ of functions. Their definition is given below.
\begin{definition}
A function $s(\cdot)$ defined and holomorphic in a domain
$\mathfrak{h}_s\subset\mathbb{D}$ belongs to the class
$S_\k(\sN_1,\sN_2)$, if the kernel
$$
K_\mu(\l)=\frac{1-s(\mu)^*s(\l)}{1-\l\overline{\mu}}
$$
 has $\k$
negative squares, i.e. for all $\l_1,...,\l_n\in\mathfrak{h}_s$ and
$u_1,...,u_n\in\sN_1$ the matrix
$([K_{\l_j}(\l_i)u_i,u_j]_{\sN_1})_{i,j=1}^n$ has at most $\kappa$
negative eigenvalues. For at least one such choice, it has exactly
$\kappa$ negative eigenvalues.
\end{definition}
In particular, if $\sN_1$ and $\sN_2$ are Hilbert spaces then the
$[\sN_1,\sN_2]$-valued function $s(\cdot)$ belongs to the class
$S(\sN_1,\sN_2)$, if the kernel $K_\m(\l)$ is positive definite
everywhere in $\dD$. As is known, the last condition is equivalent
to that $s(\cdot)$ is holomorphic in $\mathbb{D}$, and
$\|s(\l)\|\leq1$ for all $\l\in\mathbb{D}$.
\begin{proposition}\label{P:1.1}
Let $M_2(\cdot)$ be the operator-function defined by \eqref{MG2}.
Then $M_2(\cdot)\in S_\k(\sN_2,\sN_1)$.
\end{proposition}
\begin{proof}
Let $\l_j$ be some points from $\cD_2$, $j=1,...,n$. We denote
$h_j:=\G_2\wh f_{\l_j}$. Then $\G_1\wh f_{\l_j}=M_2(\l)h_j$. From
(\ref{TG}) for $\wh f_{\l_j}\in \sN_{\l_j}(V)$ and $\wh f_{\l_k}\in \sN_{\l_k}(V)$, we have
$$
(\l_j\ov\l_k-1)[f_{\l_j},f_{\l_k}]=[M_2(\l_j)h_j,M_2(\l_k)h_k]_{\sN_1}-[h_j,h_k]_{\sN_2}.
$$
Now we construct the quadratic form
$$
\sum\limits_{j,k=1}^n\sk[{\frac{I-M_2(\l_k)^\zx M_2(\l_j)}{1-\l_j\ov\l_k}h_j,h_k}]_{\sN_2}\xi_j\ov\xi_k=
\sum\limits_{j,k=1}^n[f_{\l_j},f_{\l_k}]_{\cH}\xi_j\ov\xi_k.
$$
Since $\cH$ has the negative index $\k$, and since the  reduced
quadratic form has at most $\k$ negative squares and exactly $\k$
negative squares for some collection $f_{\l_j}$,we have
$M_2(\cdot)\in S_\k(\sN_2,\sN_1)$.
\end{proof}
\begin{proposition}\label{P:1.2}
The following relations hold:
\begin{enumerate}
\item[]
\begin{equation}\label{M1}
-\frac{I-M_1(\m)^\zx M_1(\l)}{1-\l\ov\m}=\g_1(\m)^\zx \g_1(\l),\quad\l,\m\in\cD_1;
\end{equation}
\item[]
\begin{equation}\label{M2}
\frac{I-M_2(\m)^\zx M_2(\l)}{1-\l\ov\m}=\g_2(\m)^\zx \g_2(\l),\quad\l,\m\in\cD_2;
\end{equation}
\item[]
\begin{equation}\label{M3}
\frac{M_1(\m)^\zx -M_2(\l)}{1-\l\ov\m}=\g_1(\m)^\zx \g_2(\l),\quad\l\in\cD_2,\text{
}\m\in\cD_1;
\end{equation}
\item[]
\begin{equation}\label{M4}
\frac{M_1(\l)-M_2(\m)^\zx }{1-\l\ov\m}=\g_2(\m)^\zx \g_1(\l),\quad\l\in\cD_1,
\text{ }\m\in\cD_2.
\end{equation}
\end{enumerate}
\end{proposition}
\begin{proof}
We now prove \ref{M1} and \ref{M3}, because \ref{M2} is proved
analogously to \ref{M1}, and \ref{M4} is a consequence of \ref{M3}.

Let $\l,\m\in\cD_1$ and $h_1,h_1'\in\sN_1$. Then formula (\ref{MG1})
yields
$$\G\begin{bmatrix}  \g_1(\l)h_1\\\l\g_1(\l)h_1
\end{bmatrix}=\begin{bmatrix}  h_1\\M_1(\l)h_1
\end{bmatrix}\text{  and  } \G\begin{bmatrix}  \g_1(\m)h_1'\\\m\g_1(\m)h_1'
\end{bmatrix}=\begin{bmatrix}  h_1'\\M_1(\m)h_1'
\end{bmatrix}.$$
Using this identities and setting (\ref{TG}) $\wh f=\wh \g_1(\l)h_1$
and $\wh g=\wh \g_1(\m)h_1'$, we obtain
$$
(\l\ov\m-1)[\g_1(\l)h_1,\g_1(\m)h_1']_\cH=[h_1,h_1']_{\sN_1}-[M_1(\l)h_1,M_1(\m)h_1']_{\sN_2}
$$
or
$$
(\l\ov\m-1)[\g_1(\m)^\zx\g_1(\l)h_1,h_1']_\cH=[(I-M_1(\m)^\zx
M_1(\l))h_1,h_1']_{\sN_1}.
$$
From whence, we obtain equality (\ref{M1}).

Let now $\l\in\cD_2$, $\m\in\cD_1$ and let $h_1\in\sN_1$ and
$h_2\in\sN_2$. Then formulas (\ref{MG1}) and (\ref{MG2}) yield
$$\G\begin{bmatrix} \g_2(\l)h_2\\\l\g_2(\l)h_2
\end{bmatrix}=\begin{bmatrix}  M_2(\l)h_2\\h_2
\end{bmatrix}\text{  and  }\G\begin{bmatrix}  \g_1(\m)h_1\\\m\g_1(\m)h_1
\end{bmatrix}=\begin{bmatrix}  h_1\\M_1(\m)h_1
\end{bmatrix}.$$
From (\ref{TG}), we obtain
$$
(\l\ov\m-1)[\g_2(\l)h_2,\g_1(\m)h_1]_\cH=[M_2(\l)h_2,h_1]_{\sN_1}-[h_2,M_1(\m)h_1]_{\sN_2}.
$$
This yields identity (\ref{M3}).
\end{proof}
\begin{definition}
The isometric operator $V$ in $\cH$ is called simple, if
$$\ov\spn\{\sN_\l(V):\l\in\wh\r(V)\}=\cH.$$
\end{definition}
If the isometric operator $V$ in a Pontryagin space $\cH$ is simple,
then $\dD\cup\dD_e\in\wh\r(V)$ (see \cite{AI86}).

\begin{theorem}\label{T:1.1}
Let $\Pi=\{\sN_1\oplus\sN_2,\G_1,\G_2\}$ be the boundary triplet of a simple
isometric operator $V$, and let $M_1(\cdot)$ and $M_2(\cdot)$ be the functions defined by
equations (\ref{D:M1}) and (\ref{D:M2}). Then the set of poles of the operator-function $M_1(\cdot)$ in $\dD_e$
coincide with $\L_1$, and the set of poles of the operator-function $M_2(\cdot)$ in $\dD$ coincide with $\L_2$.
\end{theorem}
\begin{proof}
 It follows from (\ref{M1}) that if $\l_0$ is a pole of the operator-function
$M_1(\cdot)$, then it is a singular point for $\g_1(\cdot)$, i.e.
$\l_0\in\L_1.$

Let now $\l_0\in\L_1$. Then
$$
(V_1-\l)^{-1}=\frac{A_{-n}}{(\l-\l_0)^n}+...+\frac{A_{-1}}{\l-\l_0}+...
$$
Let us assume that $M_1(\l)$ is holomorphic at the point $\l_0$. Then the equality
$$
-\frac{I-M_1(\m)^\zx M_1(\l)}{1-\l\ov\m}=\g_1(\m)^\zx \sk({I+(\l-\m')(V_1-\l)^{-1}})\g_1(\m'),
$$
implies that $[A_{-i}\g_1(\m')h_1',\g_1(\m)h_1]=0$ for all $i=1,...,n$,
 $\m,\m'\in\cD_1$ and any $h_1,h_1'\in\sN_1$.

The equality
$$
\frac{M_1(\l)-M_2(\m)^\zx }{1-\l\ov\m}=\g_2(\m)^\zx \sk({I+(\l-\m')(V_1-\l)^{-1}})\g_1(\m'),
$$
yields $[A_{-i}\g_1(\m')h_1,\g_2(\m)h_2]=0$ for all $\m'\in\cD_1$,
$\m\in\cD_2$ and any $h_1\in\sN_1$, $h_2\in\sN_2$. By virtue of the
simplicity of the operator $V$,

$$\ov\spn\{\sN_\l(V):\l\in\cD_1\cup\cD_2\}=\cH.$$

Hence, all $A_{-i}=0$ for $i=1,...,n$. But this contradicts the
assumption that $\l_0\in\L_1$.

The second assertion of this theorem is proved analogously.
\end{proof}
Let us say
\begin{equation}\label{Mgsharp}
  M_i^\#(\l):=M_i(1/\ov\l)^{\zx}\text{ and
  }\g_i^\#(\l):=\g_i(1/\ov\l)^{\zx},
\end{equation}
where $i=1,2$.
\begin{theorem}\label{T:1.2}
Under the assumptions of Theorem \ref{T:1.1}, the equality
\begin{equation}\label{symV}
M_1(\l)=M_2^{\#}(\l)\text{ for }\l\in\cD_1
\end{equation}
holds.
\end{theorem}
\begin{proof}
Let us take $\l\in\mathbb{D}_e\backslash(\L_1\cup\L_2^\#)$, where
$\L_2^\#$ is the set symmetric to the set $\L_2$ relative to the unite disc. Setting $\m=1/\ov\l$ in (\ref{M4}),
we obtain (\ref{symV}). Equality (\ref{symV}) for $\l\in\cD_1$
can be obtained by the analytic continuation of the function $M_2^{\#}(\l)=M_1(\l)$
into the points $\l\in\cD_1\cap\L_2^\#$ and by the application of Theorem \ref{T:1.1}.
\end{proof}

\begin{remark}
By virtue of the holomorphy of $M_1(\cdot)$ in $\cD_1$ and $M_2(\cdot)$ in $\cD_2$,
 the identity proved in Theorem \ref{T:1.2} implies that if
$\l_0$ is a pole of $M_1(\cdot)$, then $1/\ov\l_0$ is a pole
 of $M_2(\cdot)$. The same is true for the poles of  $M_2(\cdot)$.
 Thus, the poles of $M_1(\cdot)$ and $M_2(\cdot)$ are symmetric
 with respect to the unit disc. Hence, $\L_1=\L_2^\#$.
\end{remark}
\begin{definition}\label{M}
The operator-function defined by the equality
\begin{equation}\label{E:M}
M(\l)=\left\{\begin{array}{cc}
                           M_1(\l), & \l\in\cD_1 \\
                         M_2(\l), & \l\in\cD_2\\
                        \end{array}\right.,
\end{equation}
is called the Weyl function of the operator $V$, which corresponds to the boundary triplet $\Pi=\{\sN_1\oplus\sN_2,\Gamma_1,\Gamma_2\}$.
\end{definition}
\begin{lemma}\label{L:1.5}
Let $V:\cH\to\cH$ be an isometric operator, and let the collection
$\Pi=\{\sN_1\oplus\sN_2,\Gamma_1,\Gamma_2\}$ be the boundary triplet
of the isometric operator $V$. Then:
\begin{enumerate}
\item
For $\l\in\cD_1$, the following equality holds
\begin{equation}\label{gamma1}
\G_2\begin{bmatrix} (V_1-\l )^{-1}
\\I+\l(V_1-\l)^{-1}
\end{bmatrix}=-\frac{1}{\l}\g_2^\#(\l);
\end{equation}
\item
for $\l\in\cD_2$, the following equality holds
\begin{equation}\label{gamma2}
\G_1\begin{bmatrix} (V_2-\l)^{-1}
\\I+\l(V_2-\l)^{-1}
\end{bmatrix}=\frac{1}{\l}\g_1^\#(\l).
\end{equation}
\end{enumerate}
\end{lemma}
\begin{proof}
 1) Take $\l\in\cD_1$,
$\m\in\cD_2 $ and $h_1\in\sN_1$. Formula (\ref{E:1.8}) yields
$$
\wh\g_1\sk({\l})h_1-\wh\g_1\sk({\m})h_1=\sk({\l-\m})
\begin{bmatrix} (V_1-\l)^{-1}
\\I+\l(V_1-\l)^{-1}
\end{bmatrix}\g_1\sk({\m})h_1.
$$
Applying the operator $\G_2$ to both sides of the equality, we obtain
$$
M_1\sk({\l})h_1-M_1\sk(\m)h_1=(\l-\m) \G_2\begin{bmatrix}
(V_1-\l)^{-1}
\\I+\l(V_1-\l)^{-1}
\end{bmatrix}\g_1\sk({\m})h_1.
$$
In this formula, we replace $M_1(\l)$ by
$M_2\sk({\frac{1}{\ov\l}})^\zx $.In view of formula (\ref{M4}), the left side can be written as follows:
$$
M_2\sk({\frac{1}{\ov\l}})^\zx h_1-M_1(\m)^\zx h_1=-\sk({1-\frac{\m}{\l}})\g_2\sk({\frac{1}{\ov\l}})^\zx \g_1(\m)h_1.
$$
Equating the right-hand sides of two last formulas, we obtain
$$
\G_2\begin{bmatrix} (V_1-\l )^{-1}
\\I+\l(V_1-\l)^{-1}
\end{bmatrix}=-\frac{1}{\l}\g_2\sk({\frac{1}{\ov\l}})^\zx .
$$
2) Take $\l\in\cD_2$, $\m\in\cD_1 $, and $h_2\in\sN_2$. Substituting
$\l$ and $\m$ in formula (\ref{E:1.8}), we write it in the form
$$
\wh\g_2(\l)h_2-\wh\g_2(\m)h_2=(\l-\m)
\begin{bmatrix} (V_2-\l)^{-1}
\\I+\l(V_2-\l)^{-1}
\end{bmatrix}\g_2\sk({\m})h_2.
$$
Applying the operator $\G_1$ to both sides of the equality, we obtain
$$
M_2(\l)h_2-M_2(\m)h_2=(\l-\m)\G_1
\begin{bmatrix} (V_2-\l)^{-1}
\\I+\l(V_2-\l)^{-1}
\end{bmatrix}\g_2\sk({\m})h_2.
$$
Replacing $M_2(\l)$ in this formula by $M_1\sk({1/\ov\l})^\zx $,
 we have
$$
M_1\sk({1/\ov\l})^\zx -M_2(\m)h_2=(\l-\m)\G_1
\begin{bmatrix} (V_2-\l)^{-1}
\\I+\l(V_2-\l)^{-1}
\end{bmatrix}\g_2\sk({\m})h_2.
$$
In view of formula (\ref{M3}),we write the left-hand side as
$$
M_1\sk({1/\ov\l})^\zx -M_2(\m)h_2=\sk({1-\frac{\m}{\l}})\g_1\sk({\frac{1}{\ov\l}})^\zx \g_2(\m)h_2.
$$
Comparing the right-hand sides of two last formulas, we obtain formula
 (\ref{gamma2}).
\end{proof}
\end{subsection}
\begin{subsection}{Description of resolvents of extensions of $V$.}

 Below, we present two theorems, describing the spectrum
 and the resolvents of extensions $V_\t$ of the operator $V$.
  The first theorem gives such a description for the points $\l$
  lying outside the unit disc $\dD$, i.e.,
$\l\in\cD_1\subset\dD_e$. Recall, that $R_\l(T)$ means the resolvent
of $T$ at $\l$ (see formula \eqref{res}).
\begin{theorem}\label{res1}
Let $V:\mathcal{H}\rightarrow\mathcal{H}$ be an isometric operator,
let the collection $\Pi=\{\sN_1\oplus\sN_2,\Gamma_1,\Gamma_2\}$ be a
boundary triplet of the isometric operator $V$, and let $\t$ be a
l.r. from $\sN_2$ to $\sN_1$. Then, for $\l\in\cD_1$ the following
assertions are valid:
\begin{enumerate}
\item[(1)]
$\l\in\s_p(V_\t)$ iff
$0\in\s_p(\t^{-1}-M_1(\l))$, in this case,
$$\ker(\t^{-1}-M_1(\l))=\sk\{{\G_1\begin{bmatrix} f \\\l f
\end{bmatrix}, f\in \ker (V_\t-\l)}\}.$$
\item[(2)]
$\l\in\r(V_\t)\cap\cD_1$ iff
$0\in\r(\t^{-1}-M_1(\l))$, for $\l\in\r(V_\t)\cap\cD_1$ the resolvent of the extension  $V_\t$ can be determined from the formula
\begin{equation}\label{E:res0}
  R_\l(V_\t)=R_\l(V_1)-\l^{-1}\g_1(\l)\sk({\t^{-1}-M_1(\l)})^{-1}\g_2^\#(\l).
\end{equation}
\end{enumerate}
\end{theorem}
\begin{proof}
1) Let $\l\in\s_p(V_\t)$ and $f_{\l}$ be an eigenvector of $V_\t$
corresponding to the eigenvalue $\l$. Hence,
$$
\begin{bmatrix}
f_{\l}\\\l f_{\l}
\end{bmatrix}\in V_\t,\text{ }
 f_{\l}\in\sN_{\l}(V), \text{ and } M_1(\l)\G_1\wh
f_{\l}=\G_2\wh f_{\l}.
$$
 Since $f_{\l}\in\dom V_\t$, we have
 $\begin{bmatrix} \G_1\wh f_{\l}\\\G_2\wh f_{\l}
\end{bmatrix}\in\t^{-1}$. Hence,
 $(\t^{-1}-M_1(\l))\G_1\wh f_{\l}=0.$

Conversely, if $(\t^{-1}-M_1(\l))h_1=0$ for some $h_1\in\sN_1$,
then the vector $f_{\l}:=\g_1(\l)h_1\in\sN_{\l}(V)$, and, hence,
$f_{\l}\in\s_p(V_\t)$.
\item
2) Assume that $0\in\r(\t^{-1}-M_1(\l))$, $\begin{bmatrix}  f\\f'
\end{bmatrix}\in V_\t$ and $g\in\cH$.
Lemma \ref{L:1.2} implies that the solution of the equation
\begin{equation}\label{E9}
f'-\l f=g
\end{equation}
can be presented in the form
\begin{equation}\label{E10}
\begin{bmatrix}  f\\f'
\end{bmatrix}=\begin{bmatrix}  f_1\\f'_1
\end{bmatrix}+\begin{bmatrix}  f_{\l}\\\l f_{\l}
\end{bmatrix}, \text{where }\wh{f}_1\in V_1,\text{
}\wh{f}_{\l}\in\wh{\sN}_{\l}(V).
\end{equation}
Then formula (\ref{E9}) yields
\begin{equation}\label{E11}
f_1=(V_1-\l )^{-1}g.
\end{equation}
Applying the operators $\G_1$ and $\G_2$ to the equality (\ref{E10}), we obtain
$$
\G_1\wh f=\G_1\wh f_{\l}
$$
$$
\G_2\wh f=\G_2\begin{bmatrix} ( V_1-\l)^{-1}g
\\g+\l( V_1-\l )^{-1}g
\end{bmatrix}+\G_2\wh f_{\l}=-\frac{1}{\l}\g_2^\#(\l) g+M_1(\l)\G_1\wh f.
$$
Since $0\in\r(\t^{-1}-M_1(\l))$, the previous equality yields
 $$\G_1\wh
f_{\l}=-\frac{1}{\l}\sk({\t^{-1}-M_1(\l)})^{-1}\g_2^\#(\l) g,$$
\begin{equation}\label{E12}
f_{\l}=-\frac{1}{\l}\g_1(\l)\sk({\t^{-1}-M_1(\l)})^{-1}\g_2^\#(\l)
g.
\end{equation}
Equalities (\ref{E10}), (\ref{E11}), and (\ref{E12}) yield equality
(\ref{E:res0}).

Conversely, let $\l\in\r(V_\t)$. By virtue of item 1), to prove the membership $0\in\r(\t^{-1}-M_1(\l))$ , it is sufficient to show that
$\ran(\t^{-1}-M_1(\l))=\sN_2$. Indeed, by virtue of the surjectivity of the mapping $\G$, there exists the vector $\wh
f_1\in V^\sx$ for any $h_2\in\sN_2$ such that $\G\wh f_1=\begin{bmatrix} 0\\h_2
\end{bmatrix}$. Since $\G_1\wh f_1=0$, we have $\wh f_1\in V_1$. We set
$f=( V_\t-\l)^{-1}(f'_1-\l f_1)$. Then
$f_{\l}:=f-f_1\in\sN_{\l}(V)$ and $f=f_1+f_{\l}$. Since $\begin{bmatrix} \G_1\wh f \\\G_2\wh f
\end{bmatrix}=\begin{bmatrix} \G_1\wh
f_{\l} \\\G_2\wh f
\end{bmatrix}\in \t^{-1}$, we obtain
$$\G_2\wh
f-M_1(\l)\G_1\wh f_{\l}=\G_2(\wh f-\wh f_{\l})=\G_2\wh f_1=h_2.$$
 This proves the equality
$\ran\sk({\t^{-1}-M_1(\l)})=\sN_2$ and also the inclusion
$0\in\r\sk({\t^{-1}-M_1(\l)})$.
\end{proof}
Let $\t$ be a closed l.r. from $\sN_2$ to $\sN_1$. Then there exists
a Hilbert space $H$ and bounded operators $K_i:H\to\sN_i$ $i=1,2$,
such that
\begin{equation}\label{theta}
\t=\sk\{{\begin{bmatrix} K_2h \\K_1h
\end{bmatrix}:\quad h\in H}\}.
\end{equation}
\begin{corollary}\label{C:res1}
If we write the l.r. $\t$ in terms of the operators  $K_1$ and $K_2$ (see
\ref{theta}), then $\l\in \r(V_\t)$ iff
$0\in\r(K_2-M_1(\l)K_1)$. Formula (\ref{E:res0}) takes the form
\begin{equation}\label{E:res1}
 R_\l(V_\t)=R_\l(V_1)-\l^{-1}\g_1(\l)K_1\sk({K_2-M_1(\l)K_1})^{-1}\g_2^\#(\l).
\end{equation}
\end{corollary}

\begin{corollary}\label{C:res}
Let $\t$ be the graph of a unitary operator $U$ from $\sN_2$ to $\sN_1$.
Then, for $\l\in\cD_1$ such that $0\in\r(I-M_1(\l)U)$, we obtain
$\l\in\r(V_\t)$, and the resolvent of an extension $V_\t$ can be found by the formula
\begin{equation}\label{E:res11}
  R_\l(V_\t)=R_\l(V_1)-\l^{-1}\g_1(\l)U\sk({I-M_1(\l)U})^{-1}\g_2^\#(\l).
\end{equation}
\end{corollary}
The following result for the point $\l$ inside the unit disc $\dD$ can be proved analogously.
\begin{theorem}\label{res00}
Let $V:\mathcal{H}\rightarrow\mathcal{H}$ be an isometric operator,
let the collection $\Pi=\{\sN_1\oplus\sN_2,\Gamma_1,\Gamma_2\}$ be
 the boundary triplet of an isometric operator $V$, and $\t$ be
 a l.r. from $\sN_2$ to $\sN_1$. Then, for $\l\in\cD_2$, the following assertions are true:
\begin{enumerate}
\item[(1)]
$\l\in\s_p(V_\t)$ iff $0\in\s_p(\t-M_2(\l))$, in this case,
$$\ker(\t-M_2(\l))=\sk\{{\G_2\begin{bmatrix} f \\\l f
\end{bmatrix}, f \in\ker (V_\t-\l)}\}.$$
\item[(2)]
$\l\in\r(V_\t)$ iff $0\in\r(\t-M_2(\l))$;
for $\l\in\r(V_\t)\cap\cD_2$ the resolvent of an extension $V_\t$ can be found by the formula
\begin{equation}\label{E:res00}
  R_\l(V_\t)=R_\l(V_2)+\l^{-1}\g_2(\l)\sk({\t-M_2(\l)})^{-1}\g_1^\#(\l).
\end{equation}
\end{enumerate}
\end{theorem}
\begin{remark}\label{connection}
  Note that formula \eqref{E:res00} can be obtained from formula \eqref{E:res0} by taking "sharp" of both sides and using the formulas $V_1^{-[*]}=V_2$ and $M_1(\l)^\#=M_2(\l)$. Indeed,
  \begin{equation*}
R_{1/\ov\l}(V_\t)^{[*]}=R_{1/\ov\l}(V_1)^{[*]}
-\l\g_2(\l)\sk({\t^{-1}-M_1(1/\ov\l)})^{-[*]}\g_1^\#(\l),
\end{equation*}
\begin{equation*}
 \begin{split}
  -\l I-\l^2 R_\l(V_{\t^{-[*]}})=&-\l I-\l^2R_\l (V_2)
  -\l\g_2(\l)\sk({\t^{-1}-M_1^\#(\l)})^{-1}\g_1^\#(\l).
 \end{split}
\end{equation*}
After simplifications one gets
\begin{equation*}
R_\l(V_{\t^{-[*]}} )=R_\l(V_2
)+\l^{-1}\g_2(\l)\sk({\t^{-[*]}-M_2(\l)})^{-1}\g_1^\#(\l).
\end{equation*}
\end{remark}
\begin{corollary}\label{C:res2}
If the l.r. $\t$ is written in terms of the operators $K_1$ and
$K_2$ (see (\ref{theta})), then formula (\ref{E:res00}) takes the
form
\begin{equation}\label{E:res2}
   R_\l(V_\t)=R_\l(V_2)+\l^{-1}\g_2(\l)K_2\sk({K_1-M_2(\l)K_2})^{-1}\g_1^\#(\l).
\end{equation}
\end{corollary}

\end{subsection}
\section{Linear fractional transformations of an isometric operator.}
Let $z_0(\in \dD_e)$ be a regular type point for an isometric
operator $V$ and let $\w V$ be an extension of $V$ such that $z_0\in
\r(\w V)$. Then the operators
\begin{equation}\label{E:LFT}
 V_0:=(I-|z_0|^2 )(V-z_0I)^{-1}-\ov z_0 I
\end{equation}
\begin{equation}\label{E:LFT2}
 \w V_0:=(I-|z_0|^2 )(\w V-z_0I)^{-1}-\ov z_0 I
\end{equation}
are well defined.
\begin{lemma}\label{L:Rs}
 The resolvent set of $\w V_0$, i.e. the set of $\zeta\in\mathbb{C}$
  such that $(\w V_0-\zeta I)$ is boundary invertible, is connected
  to the resolvent set of the l.r. $\w V$ by the formula
 \begin{equation}\label{E:l}
  \zeta=\frac{1-\ov z_0\l}{\l-z_0}, \text{ where } \l\in\r(\w V).
 \end{equation}
\end{lemma}
\begin{proof}
 Indeed, the operator
\begin{equation}\label{regulpoint}
  \begin{split}
    \w V_0-\zeta I&=(I-|z_0|^2 )(\w V-z_0I)^{-1}-(\ov z_0+\zeta) I\\
    &=\frac{|z_0|^2-1}{\l-z_0}\sk({I-(\l-z_0)(\w V-z_0I)^{-1}})
  \end{split}
  \end{equation}
 is invertible if and only if $\l=\frac{1+\zeta z_0}{\ov z_0+\zeta}\in\r(\w V)$. Then $\zeta=\frac{1-\ov z_0 \l}{\l-z_0}$.
\end{proof}

\begin{lemma}
  The connection between the resolvents of $\w V_0$ and $\w V$ are the following
  \begin{equation}\label{RR'}
   (\w V_0-\zeta I)^{-1}=\frac{\l-z_0}{|z_0|^2-1}\sk({I+(\l-z_0)(\w V-\l I)^{-1}}).
  \end{equation}

\end{lemma}
\begin{proof}
Namely it follows from \eqref{regulpoint} that
\begin{equation*}
 \begin{split}
  (\w V_0-\zeta I)^{-1}&=\frac{|z_0|^2-1}{\l-z_0}\sk({I-(\l-z_0)(\w V-z_0I)^{-1}})^{-1}\\
  &=\frac{\l-z_0}{|z_0|^2-1}\sk({I+(\l-z_0)(\w V-\l I)^{-1}}).\qedhere
 \end{split}
\end{equation*}
\end{proof}

In the next Lemma connections between boundary operators,
$\g$-fields and Weyl-functions of $V_0$ and $V$ will be established.

\begin{lemma}\label{L:connect}
Let $V$ be an isometric operator, let $z_0 \in\wh \r(V)\cap\dD_e$
and let $V_0$ and $\zeta$ be given by \eqref{E:LFT} and \eqref{E:l},
and let $\Pi=\{\sN_1\oplus\sN_2,\G_1,\G_2\}$ be a boundary triplet
for $V$. Then:

(1) The linear relation $(V_0^\zx)^{-1}$ takes the form
\begin{equation}\label{Vsx}
  (V_0^\zx)^{-1}=\sk\{{\wh f=
  \begin{bmatrix} h'-z_0h\\h-\ov z_0 h'\end{bmatrix}:
  \wh h\in V^\sx}\};
\end{equation}

(2) A boundary triplet $\Pi^0=\{\sN_1\oplus\sN_2,\G_1^0,\G_2^0\}$
for $V_0^\sx$ can be given by the formulas
\begin{equation}\label{E:G}
  \G_j^0\wh f=\sqrt{|z_0|^2-1}\G_j \wh h\quad(j=1,2)
\end{equation}
where $\wh f$ and $\wh h$ are connected as in \eqref{Vsx}.

(3) The Weyl functions $M_j(\zeta)$ and $M_j^0(\l)$ and the
$\g$-fields $\g_j(\zeta)$ and $\g_j^0(\l)$ corresponding to the
boundary triplets $\Pi$ and $\Pi^0$, respectively, are connected by
the formulas
\begin{equation}\label{GMF}
  M^0_j(\zeta)=M_j(\l),\quad
  \g^0_j(\zeta)=\frac{\l-z_0}{\sqrt{|z_0|^2-1}}\g_j(\l)\quad
  (j=1,2).
\end{equation}
\end{lemma}
\begin{proof}
(1)  By using the formula \eqref{E:LFT} one obtains
 \begin{equation}\label{Vzx}
  V_0^\zx=(1-|z_0|^2)\sk({V^\zx-\ov z_0})^{-1}-z_0 I.
 \end{equation}
 The formula \eqref{Vsx} follows from \eqref{Vzx}.

(2) Next it follows from \eqref{E:G}, \eqref{TG}, and \eqref{Vsx}
that
\begin{equation*}
 \begin{split}
  &\sk[{\G^0_1 \wh f,\G^0_1 \wh f}]_{\sN_1}-\sk[{\G^0_2 \wh f,\G^0_2 \wh
  f}]_{\sN_2}
  =(|z_0|^2-1)\sk({\sk[{\G_1 \wh h,\G_1 \wh h}]_{\sN_1}-\sk[{\G_2 \wh
h,\G_2 \wh h}]_{\sN_2}})\\
  &=(|z_0|^2-1)\sk({[h',h']_\cH-[h,h]_\cH})
  =[h-\ov z_0h',h-\ov
z_0h']_\cH-[h'-z_0h,h'-z_0h]_\cH\\
&=[f',f']_{\cH}-[f,f]_\cH.
 \end{split}
\end{equation*}

(3) Let $\wh h=\begin{bmatrix}   h\\h'
\end{bmatrix}\in \sN_\l(V)$, i.e., $\wh h\in V^\sx$ and $h'=\l
h$. Then
\begin{equation}
  \wh f =\begin{bmatrix}
    f\\f'
  \end{bmatrix}:=\begin{bmatrix}\label{defvec}
    h'-z_0h\\h-\ov z_0 h'
  \end{bmatrix}\in \sN_\zeta(V_0).
\end{equation}
Let us set $u_j:=\G_j\wh h$ $(j=1,2)$. Then by \eqref{E:1.7}
$\wh\g_j(\l)u_j=\wh h$. It follows from \eqref{E:G} that
$$
 \G_j^0\wh f=\sqrt{|z_0|^2-1}\G_j\wh h=\sqrt{|z_0|^2-1}u_j
$$
and $\wh\g_j^0(\zeta)u_j=\frac{1}{\sqrt{|z_0|^2-1}}\wh h$.
 Hence by \eqref{E:1.7} and \eqref{defvec}
$$
 \g_j^0(\zeta)u_j=\frac{1}{\sqrt{|z_0|^2-1}}\pi_1\wh f=\frac{\l-z_0}{\sqrt{|z_0|^2-1}}h=\frac{\l-z_0}{\sqrt{|z_0|^2-1}}\g_j(\l)u_j.
$$
Since \eqref{MG1} and \eqref{MG2} one gets
\begin{equation*}
  \begin{split}
    M_j^0(\zeta)u_j&=\G_i^0\wh \g_j^0(\zeta)u_j=\frac{1}{\sqrt{|z_0|^2-1}}\G_i^0\wh f
    =\G_i\wh h=\G_i\wh\g_j(\l)u_j=M_j(\l)u_j,
  \end{split}
\end{equation*}
where $i,j=1,2$ and $i\neq j$.
\end{proof}

\section{Description of generalized resolvents.}
\begin{definition}
(see \cite{L71}) The operator-function $\mathbb{R}_\l$ holomorphic
in neighborhood $\cO$ of the point $\zeta\in \cD_1$  is called the
generalized resolvent of an isometric operator
$V:\mathcal{H}\rightarrow\mathcal{H}$, if there exist a Pontryagin
space $\w{\cH}\supset\cH$ and a unitary extension
$\w{V}:\w{\mathcal{H}}\rightarrow\w{\mathcal{H}}$ of the operator
$V$ such that $\l\in\r(\w{V})$, and if the equality
\begin{equation}\label{minA}
\mathbf{R}_\l=P_{\cH}\sk({\w{V}-\l})^{-1}\upharpoonright\cH,
\quad\l\in\r(\w{V})\cap\cO
\end{equation}
in which $P_{\cH}$ is the orthoprojector from $\w{\cH}$ onto $\cH$ holds.
\end{definition}
\begin{definition}
A unitary extension $\w V$ of an operator $V$ is called minimal, if
 $\cH_{\w V}=\w\cH$, where
\begin{equation}\label{min}
\cH_{\w V}:=\overline{\spn}\sk\{{\cH+(\w
V-\l)^{-1}\cH:\quad\l\in\r(\w V)}\}.
\end{equation}
\end{definition}
\begin{definition}
A unitary extension $\w V$ of an operator $V$ is called regular, if
$\w \cH_{\w V}^{[\p]}:=\w\cH[-] \cH_{\w V}$ is a Hilbert space.
\end{definition}
\begin{proposition}\label{P:regmin}
Let a regular extension $\w V$ of an operator $V$ be not minimal.

Then the following decompositions are valid:
\begin{equation}\label{razl0}
\w\cH=\cH^{[\p]}_{\w V}[\dotplus]\cH_{\w V}\text{ and }\w V= \w
V_u[\dotplus]\w V_m.
\end{equation}
 Here $\w V_m$ is the minimal extension of the operator $V$ and $\w V_u$
 is a unitary operator in a Hilbert space $\cH^{[\p]}_{\w V}$.
In this case,
\begin{equation}\label{min2}
P_\cH(\w V-\l)^{-1}\upharpoonright\cH=P_\cH(\w
V_m-\l)^{-1}\upharpoonright\cH.
\end{equation}
\end{proposition}
\begin{proof}
Since $\w V$ is the regular extension, we have $\ind_-\cH_{\w
V}=\ind_-\w\cH=\w\k$. Hence, $\cH_{\w V}$ is not degenerate.

We now show that $\cH_{\w V}$ and $\cH_{\w V}^{[\p]}$ are invariant
for $\w V$. Let us take different $\l_1$ and $\l_2$ from $\r(\w V)$.
Let $h\in\cH$. Then $u:=(\w V-\l_2 )^{-1}h\in\cH_{\w V}$. Let the
operator $(\w V-\l_1)^{-1}$ act on this vector:
$$
(\w V-\l_1 )^{-1}(\w V-\l_2 )^{-1}h=\frac{1}{\l_1-\l_2}\sk({(\w
V-\l_1 )^{-1}-(\w V-\l_2 )^{-1}})h\in\cH_{\w V}.
$$
The case where $\l_1$ and $\l_2$ coincide with each other follows
from the previous one, if $\l_1$ tends to $\l_2$.

Consider now the vectors $v\in\cH_{\w V}^{[\p]}$ and $u\in\cH_{\w
V}$. Then
\begin{equation*}
 \begin{split}
  &\sk[{(\w V-\l)^{-1}v,u}]_{\w\cH}=\sk[{v,(\w
  V^*-\ov\l)^{-1}u}]_{\w\cH}
  =\sk[{v,\frac{1}{\ov\l}\sk({-I+(I-\ov\l\w
  V)^{-1}})u}]_{\w\cH}=0.
 \end{split}
\end{equation*}
 Here, we use the fact that, for the unitary operator $\w V$, the
inclusion $\l\in\r(\w V)$ yields the inclusion
$\frac{1}{\ov\l}\in\r(\w V)$.

Thus, $\w \cH=\cH_{\w V}^{[\p]}[\dotplus]\cH_{\w V}$ and $\w
V=\begin{bmatrix} \w V_u&0
\\0&\w V_m
\end{bmatrix}$, where $\w V_m$ is the minimal extension of the operator $V$ in $\cH_{\w V}$.

The equality
$$
P_\cH(\w V-\l)^{-1}\upharpoonright\cH=P_\cH(\w
V_m-\l)^{-1}\upharpoonright\cH,
$$
follows from representation (\ref{razl0}).
\end{proof}

\begin{theorem}\label{GRes1}
Let $V$ be an isometry in a Pontryagin space $\cH$ with negative
index $\k$, let $\Pi=\{\sN_1\oplus\sN_2,\Gamma_1,\Gamma_2\}$ be the
boundary triplet for  $V$, $V_i=\ker\G_i$, and $\g_i(\cdot)$,
$M_i(\cdot)$, $i=1,2$ be the corresponding $\g$-fields and the Weyl
functions and the condition (A) holds.

Let $\w\cH=\cH^{[\p]}[\dotplus]\cH$ be a Pontryagin space
\\$\ind_-\w\cH=\w\k\geq\k$, $\ind_-\cH^{[\p]}=\w\k-\k$.We define the projectors $\pi_1$ and $\pi_2$
from $\cH^{[\p]}\times\cH^{[\p]}$ onto the first and second
components $\cH^{[\p]}\times\cH^{[\p]}$,

$$
\pi_1 \wh h=h,\quad\pi_2 \wh h=h',\text{ where }\wh
h=\begin{bmatrix} h
\\h'
\end{bmatrix}\in(\cH^{[\p]})^2.
$$

Then
\begin{enumerate}
\item[(1)]
the adjoint l.r. for $V^{-1}$ in the space $\w\cH$ takes the form
\begin{equation}\label{17}
V^\sx_{\w\cH}=(\cH^{[\p]})^2[\dotplus]V^\sx;
\end{equation}
\item[(2)]
the operators
\begin{equation}\label{G11} \w\G_1=\begin{bmatrix}
\pi_2&0 \\0&\G_1
\end{bmatrix}\in[(\cH^{[\p]})^2[\dotplus]V^\sx,\cH^{[\p]}[\dotplus]\sN_1].
\end{equation}
\begin{equation}\label{G22}
\w\G_2=\begin{bmatrix} \pi_1&0 \\0&\G_2
\end{bmatrix}\in[(\cH^{[\p]})^2[\dotplus]V^\sx,\cH^{[\p]}[\dotplus]\sN_2].
\end{equation}
are the boundary operators in the boundary triplet
$$\w\Pi=\{(\cH^{[\p]}[\dotplus]\sN_1)\oplus(\cH^{[\p]}[\dotplus]\sN_2),\w\Gamma_1,\w\Gamma_2\}$$
for the isometry $V$ in $\w\cH$.

Moreover,
\begin{equation}\label{V12}
 \begin{split}
  \w V_1(=\ker\w\G_1)&=(\cH^{[\p]}\oplus\{0\}) [\dotplus]V_1,\\
  \w V_2(=\ker\w\G_2)&=(\{0\}\oplus\cH^{[\p]})[\dotplus]V_2
 \end{split}
\end{equation}
and the corresponding $\g$--fields and the Weyl functions for the
boundary triplet $\w\Pi$ take the form
\begin{equation}\label{gamma12}
 \begin{split}
  \w\g_1(\l)&=\begin{bmatrix}\frac{1}{\l}I_{\cH^{[\p]}}&0
  \\0&\g_1(\l)\end{bmatrix},\quad\l\in\cD_1\\
  \w\g_2(\l)&=\begin{bmatrix}I_{\cH^{[\p]}}&0
  \\0&\g_2(\l)\end{bmatrix},\quad\l\in\cD_2
 \end{split}
\end{equation}
\begin{equation}\label{M12}
 \begin{split}
  \w M_1(\l)&=\begin{bmatrix} \frac{1}{\l}I_{\cH^{[\p]}}&0
  \\0&M_1(\l)\end{bmatrix},\quad\l\in\cD_1\\
  \w M_2(\l)&=\begin{bmatrix}\l I_{\cH^{[\p]}}&0
  \\0&M_2(\l)\end{bmatrix},\quad\l\in\cD_2
  \end{split}
\end{equation}
\end{enumerate}
\end{theorem}

\begin{proof}
Equality \eqref{17} is obvious. Let us prove equalities \eqref{G11} and
\eqref{G22}. Suppose
$\wh f=
 \begin{bmatrix}
  f\\
  f'
 \end{bmatrix}
$, $\wh g=
 \begin{bmatrix}
  g\\
  g'
 \end{bmatrix}
 \in V^\sx
$ and $ \wh m=
\begin{bmatrix}
 m\\
 m'
\end{bmatrix}
$, $ \wh n=
\begin{bmatrix}
 n\\
 n'
\end{bmatrix}\in (\cH^{[\p]})^2
$. Then $\wh f+\wh m$ and $\wh g+\wh n \in
V^\sx[\dotplus](\cH^{[\p]})^2$. Let us check the general Green
equality
\begin{equation*}
\begin{split}
&[f',g']_{\cH}-[f,g]_{\cH}+[m',n']_{\cH^{[\p]}}-[m,n]_{\cH^{[\p]}}\\
&=\sk[{\G_1\wh f,\G_1\wh g}]_{\sN_1}-\sk[{\G_2\wh f,\G_2\wh
g}]_{\sN_2}+\sk[{\pi_2\wh m,\pi_2\wh n}]_{\cH^{[\p]}}-\sk[{\pi_1\wh m,\pi_1\wh n}]_{\cH^{[\p]}}\\
&=\sk[{\w\G_1(\wh f+\wh m),\w\G_1(\wh g+\wh
n)}]_{\cH^{[\p]}[\dotplus]\sN_1}- \sk[{\w\G_2(\wh f+\wh
m),\w\G_2(\wh g+\wh n)}]_{\cH^{[\p]}[\dotplus]\sN_2}.
\end{split}
\end{equation*}
The defect space takes the form
$$
\wh\sN_\l(V)_{\w\cH}=\sk\{{\wh f_\l+\wh g:\ \wh f_\l\in
\wh\sN_\l(V),\ \wh g=\begin{bmatrix}
g \\
\l g
\end{bmatrix}\in\begin{bmatrix}
 \cH^{[\p]}\\
\cH^{[\p]}
\end{bmatrix}}\}.
$$
Formulas \eqref{V12} have become obvious now. Prove the formulas of
$\g$-fields. Taking into account
$$
\w\G_1(\wh f_\l+\wh g)=\l g+\G_1\wh f_\l,
$$
we obtain
$$
\w\g_1(\l)=\begin{bmatrix}\frac{1}{\l}I_{\cH^{[\p]}}&0
\\0&\g_1(\l)\end{bmatrix}.
$$
Similarly, we get the formula of $\w\g_2(\cdot)$ from
$$
\w\G_2(\wh f_\l+\wh g)=g+\G_2\wh f_\l.
$$
Finally let us prove the formulas for Weyl functions. Take $h_1\in\sN_1$ and
$g\in\cH^{[\p]}$.Then be definition of Weyl function one gets
$$
\w M_1(\l)(h_1+g)=\w \G_2\wh{\w\g}_1(\l)(h_1+g),
$$
where $\wh{\w\g}_1(\l)=\begin{bmatrix}
 \w\g_1(\l)\\
\l \w\g_1(\l)
\end{bmatrix}$.
We obtain
\begin{equation*}
\begin{split}
\w M_1(\l)(h_1+g)&=\w \G_2\sk({\begin{bmatrix}
 \frac{1}{\l}I_{\cH^{[\p]}}\\
I_{\cH^{[\p]}}
\end{bmatrix}g+\wh \g_1(\l)h_1})=\frac{1}{\l}g+M_1(\l)h_1.
\end{split}
\end{equation*}
Similarly, we obtain the latter formula. Now suppose $h_2\in\sN_2$. Then
\begin{equation*}
\begin{split}
\w M_2(\l)(h_2+g)&=\w \G_1\wh{\w\g}_2(\l)(h_2+g)=\w
\G_1\sk({\begin{bmatrix}
 I_{\cH^{[\p]}}\\
\l I_{\cH^{[\p]}}
\end{bmatrix}g+\wh \g_2(\l)h_2})\\&=\l g+M_2(\l)h_1.\qedhere
\end{split}
\end{equation*}
\end{proof}

We recall the basic notions of the theory of unitary colligations
(see \cite{ADRS97}, \cite{Br}). Let $\cH$,$\sN_2$ and $\sN_1$ be a
Pontryagin spaces, and let
$U=\begin{bmatrix} T & F\\
                           G & H\end{bmatrix} $
be a unitary operator from $\cH[\dotplus]\sN_2$ to
$\cH[\dotplus]\sN_1$. Then the quadruple $\Delta=(\cH,\sN_2,\sN_1;
U)$ is called a unitary colligation. The spaces  $\cH,\sN_2,\sN_1$
are called, respectively, the space of states, space of inputs, and
space of outputs, and the operator $U$ is called the connecting
operator of the colligation $\Delta$.

The colligation $\Delta$ is called simple, if there exists no subspace in the space  $\cH$ reducing $U$. The operator-function
\begin{equation}\label{ChF}
\Theta(\l)=H+\l G(I-\l T)^{-1}F:\sN_2\to\sN_1\quad
(\l^{-1}\in\rho(T))
\end{equation}
is called the characteristic function of a colligation $\Delta$ (or the scattering matrix of the unitary operator $U$ relative to the channel spaces $\sN_2$ and $\sN_1$ in the case where $\sN_2,\sN_1,\cH$ are Hilbert ones \cite{ArGr83}). The characteristic function characterizes a simple unitary colligation to within a unitary equivalence.

We recall that for the components of a unitary colligation the following relations
\begin{equation}\label{E:komp}
\begin{split}
&T^\zx T+G^\zx G=I_{\cH},\ F^\zx F+H^\zx H=I_{\sN_2},\ T^\zx F+G^\zx
H=0,\\& TT^\zx+FF^\zx=I_{\cH},\ GG^\zx+HH^\zx=I_{\sN_1},\
TG^\zx+FH^\zx=0
\end{split}
\end{equation}
hold.
\begin{proposition}\cite{D2001}
Let $\Delta=(\cH, \sN_2,\sN_1;T,F,G,H)$ be a unitary colligation and
 $\T(\cdot)$ be the characteristic function of this colligation . Then
\begin{eqnarray}\label{harfun}
\T(\l)=P_{\sN_1}(I-\l
UP_{\cH})^{-1}U\upharpoonright\sN_2=P_{\sN_1}U(I-\l
P_{\cH}U)^{-1}\upharpoonright\sN_2,
\end{eqnarray}
where $P_{\cH}$ and $P_{\sN_i}$ are orthoprojections from
$\cH[\dotplus]\sN_i$ onto $\cH$ and $\sN_i$ $(i=1,2)$, respectively.
\end{proposition}
\begin{proof}
Indeed, by the equality
\[
(I_{\cH\oplus\sN_2}-\l P_\cH U)^{-1}=
\begin{bmatrix} (I_{\cH}-\l T)^{-1} & \l (I_{\cH}-\l T)^{-1}F\\
                           0 & I_{\sN_2}
                           \end{bmatrix},
\]
we get
\begin{equation}\label{blok-matr}
U(I-\l P_\cH U)^{-1}=\begin{bmatrix} T(I-\l T)^{-1} & F+\l T(I-\l T)^{-1}F\\
                           G(I-\l T)^{-1} & H+\l G(I-\l T)^{-1}F\end{bmatrix}
\end{equation}
and the bottom right corner coincides with $\T(\l)$. This proves the first equation in ~\eqref{harfun}.

Further, note that $U^{-1}=U^\zx$ and
\begin{equation*}
\begin{split}
  U(I-\l P_\cH U)^{-1}=(U^{-1}-\l P_\cH)^{-1}=\left(U^{-1}(I-\l UP_\cH)\right)^{-1}=(1-\l UP_\cH)^{-1}U,
\end{split}
\end{equation*}
we obtain the second and the third equalities for $\T(\l)$.
\end{proof}

\begin{theorem}\label{GRes2}
Let $V$ be an isometric operator in $\cH$, let
$\w\cH=\cH[\dotplus]\cH^{[\p]}$ be a Pontryagin space with negative
index $\ind_-\w\cH=\w\k$, $\ind_-\cH=\k$, and let $\w\Pi$ be the
boundary triplet constructed in Theorem \ref{GRes1}. Then:
\begin{enumerate}
\item
Any unitary extension $\w V\in\mathcal{C}(\w\cH)$ of the operator
$V$ can be represented in the form $\w V=\w V_\t:=\w\G^{-1}\t^{-1}$,
where $\t$ is the graph of the unitary operator
\begin{equation}\label{U}
 U=\begin{bmatrix} T&F\\G&H
\end{bmatrix}:\begin{bmatrix}  \cH^{[\p]}\\\sN_2
\end{bmatrix}\to\begin{bmatrix} \cH^{[\p]}\\ \sN_1
\end{bmatrix}.
\end{equation}
\item
A unitary extension $\w V_\t\in\mathcal{C}(\w\cH)$ of the operator
$V$ is minimal iff the  unitary colligation
$\Delta=(\cH^{[\p]},\sN_2,\sN_1;T,F,G,H)$ is simple.
\item
If $\T(\l)$ is the characteristic function of the unitary
colligation
\\$\Delta=(\cH^{[\p]},\sN_2,\sN_1;T,F,G,H)$, then the generalized resolvent of the operator $V$,
which corresponds to the extension $\w V_\t$, takes the following
form for $\l\in\r(\w V)\cap\cD_1$:
\begin{equation}\label{gres1}
\mathbf{R}_\l=R_\l(V_1)-\l^{-1}\g_1(\l)\T(1/\l)\sk({I-M_1(\l)\T(1/\l)})^{-1}\g_2^\#(\l);
\end{equation}
and if  $\l\in\r(\w V)\cap\cD_2$, it takes the form
\begin{equation}\label{gres2}
\mathbf{R}_\l=R_\l(V_2)
+\l^{-1}\g_2(\l)\T(\ov\l)^\zx\sk({I-M_2(\l)\T(\ov\l)^\zx})^{-1}\g_1^\#(\l).
\end{equation}
\end{enumerate}
\end{theorem}
\begin{proof}
1) The assertion of this item of the theorem is a consequence of
Theorem \ref{T:1}.3 and the remark after it.

2) Let the colligation $\Delta=(\cH^{[\p]},\sN_2,\sN_1;T,F,G,H)$ be not
simple, i.e., $\cH^{[\p]}=\cH^{[\p]}_1[\dotplus]\cH^{[\p]}_2$. Then the unitary operator $U$
takes the form $U=\begin{bmatrix} U_1&0
\\0&U_2
\end{bmatrix}:\begin{bmatrix} \cH^{[\p]}_1 \\\cH^{[\p]}_2[\dotplus]\sN_2
\end{bmatrix}\to\begin{bmatrix} \cH^{[\p]}_1 \\\cH^{[\p]}_2[\dotplus]\sN_1
\end{bmatrix}$. In view of operators $\w\G_1$ and $\w\G_2$ (see formulas \ref{G11}) and \ref{G22}),we can conclude that they act from $\begin{bmatrix} \cH^{[\p]}_1
\\\cH^{[\p]}_2
\end{bmatrix}$ to $\cH^{[\p]}_1$ as projections. Hence, $\w V=V_\t$ will have a reducing subspace, namely, $\cH^{[\p]}_1$. Thus,
$\w V$ is not the minimal extension of the operator $V$ in $\w\cH$. The proof of this assertion in the reverse direction is analogous.
 reverse direction is analogous.

 3) Using formulas (\ref{E:res1}) and
(\ref{E:res2}) for the resolvents of extensions of the operator $V$,
we now find the resolvent of the unitary extension $\w V=\w
V_\t:\w\cH\to\w\cH$, where
$$\t=\sk\{{\begin{bmatrix}
 h\\
Uh
\end{bmatrix}: h\in\begin{bmatrix} \cH^{[\p]}\\\sN_2
\end{bmatrix}}\}.$$
Then with regard for (\ref{minA}), we obtain
\begin{equation}\label{18}
\begin{split}
  \mathbf{R}_\l g&=P_\cH R_\l(\w V_1)g-\frac{1}{\l}
  P_\cH\w\g_1(\l)U\sk({I-\w M_1(\l)U})^{-1}\w\g_2^\#(\l) g\\
  &=R_\l(V_1)g-\frac{1}{\l}
  \g_1(\l)P_{\sN_1}U\sk({I-\w M_1(\l)U})^{-1}\g_2^\#(\l) g,
\end{split}
\end{equation}
where $\l\in\r(\w V_\t)\cap\cD_1$ and $g\in\cH$;
\begin{equation*}
\begin{split}
  \mathbf{R}_\l g&=P_\cH R_\l(\w V_2)g+\frac{1}{\l}
  P_\cH\w\g_2(\l)U^\zx\sk({I-\w M_2(\l)U^\zx})^{-1}\w\g_1^\#(\l) g\\
  &=R_\l(V_2)g+\frac{1}{\l}
  \g_2(\l)P_{\sN_1}U^\zx\sk({I-\w M_2(\l)U^\zx})^{-1}\g_1^\#(\l) g,
\end{split}
\end{equation*}
where $\l\in\r(\w V_\t)\cap\cD_2$ and $g\in\cH$.

 The latter formula includes $U^\zx$, since
 $$\t=\sk\{{\begin{bmatrix}  h\\Uh
\end{bmatrix},h\in\begin{bmatrix}  \cH^{[\p]}\\\sN_2
\end{bmatrix}}\}=\sk\{{\begin{bmatrix}  U^\zx g\\g
\end{bmatrix}, g\in\begin{bmatrix}  \cH^{[\p]}\\\sN_1
\end{bmatrix}}\}$$
by virtue of the unitary of the operator $U$.

Using the Frobenius formula for the inverse block matrix, we
transform the former formula as
\begin{equation}\label{Fr}
\begin{split}
\sk({I-\w M_1(\l)U})^{-1}&= \sk({I-\begin{bmatrix} \frac{1}{\l}
I_{\cH^\p}&0
\\0&M_1(\l)\end{bmatrix}\begin{bmatrix} T&F\\G&H
\end{bmatrix}})^{-1}
\\&=\begin{bmatrix}I-\frac{1}{\l} T &-\frac{1}{\l} F
\\-M_1(\l)G&I-M_1(\l)H
\end{bmatrix}^{-1}\\&=\begin{bmatrix}\ast & \frac{1}{\l}(I- \frac{1}{\l} T)^{-1}F\Phi( \frac{1}{\l})\quad
\\\ast&\Phi( \frac{1}{\l})\quad
\end{bmatrix},
\end{split}
\end{equation}
 where $\ast$ stands for the blocks, which are insignificant, and
\begin{equation}\label{FF}
 \begin{split}
\Phi\sk({ \frac{1}{\l}}):&=\sk({I-M_1(\l)H- \frac{1}{\l}
M_1(\l)G\sk({I- \frac{1}{\l}
T})^{-1}F})^{-1}\\
&=\sk({I-M_1(\l)\T\sk({ \frac{1}{\l}})})^{-1}.
 \end{split}
\end{equation}
 Substituting (\ref{Fr}) and (\ref{FF}) in (\ref{18}), we obtain
\begin{equation*}
\begin{split}
P_{\sN_1}&U\sk({I-\w M_1(\l)U})^{-1}\upharpoonright_{\sN_2}=
\frac{1}{\l} G(I-\frac{1}{\l} T)^{-1}F\Phi(\frac{1}{\l})+H\Phi(\frac{1}{\l})\\
&=\T\sk({\frac{1}{\l}})\Phi(\frac{1}{\l})=\T\sk({\frac{1}{\l}})\sk({I-M_1(\l)\T\sk({\frac{1}{\l}})})^{-1}.
\end{split}
\end{equation*}
The latter formula is gotten by using Remark \ref{connection} and formula $\T^\#\sk({\frac{1}{\l}})=\T(\ov \l)^\zx$.
\end{proof}
\begin{theorem}\label{C:gres}
Let $V$ be an isometric operator in $\cH$ and
$\Pi=\{\sN_1\oplus\sN_2,\Gamma_1,\Gamma_2\}$ be a boundary triplet
for it such that condition (A) holds. Then formulas (\ref{gres1})
and (\ref{gres2}) establish the bijective correspondence between the
set of generalized resolvents of the operator $V$ that correspond to
regular extensions $\w V:\w\cH\to\w\cH$
 and the set of operator-functions $\e(\cdot)\in S_{\w\k-\k}(\sN_2,\sN_1)$.
\end{theorem}
\begin{proof}
Step 1. By virtue of Proposition \ref{P:regmin} for any general
extension $\w V$ of $V$ exists a minimal extension $\w V_m$, such
that the generalized resolvent for $\mathbf{R}_\l$ that corresponds
to $\w V$ admits the following representation
$\mathbf{R}_\l=P_{\cH}\sk({\w{V}_m-\l})^{-1}\upharpoonright\cH$.
Then from Theorem \ref{GRes2} the operator $\w V_m$ has the form $\w
V_m=\w \G^{-1}\t^{-1}$, where $\t$ is the graph of a unitary
operator $U$ from (\ref{U}) and the unitary colligation
$\{U,\cH^{[\p]}[\dotplus]\sN_2,\cH^{[\p]}[\dotplus]\sN_1\}$ is
simple. Let $\T(\cdot)$ be the characteristic function for this
colligation. Using Theorem \ref{GRes2}.3 we get formulas
(\ref{gres1}), (\ref{gres2}). By Theorem 2.5.10 \cite{ADRS97} the
operator-function $\T(\cdot)\in S_{\w\k-\k}(\sN_2,\sN_1)$.

To prove the opposite statement we consider separately two cases:
when $\e(\cdot)$ is holomorphic at zero and when it is not.

Step 2. Assume that $\e(\cdot)\in S_{\w\k-\k}(\sN_2,\sN_1)$ and
$\e(\cdot)$ is holomorphic at zero. Then there exists the unique
simple unitary colligation to within a unitary equivalence
$\Delta=\{\cH^{[\p]},\sN_2,\sN_1;U\}$ such that $\cH^{[\p]}$ is a
Pontryagin space with the index $\w\k-\k$ and the characteristic
function of this colligation coincide with $\e(\cdot)$ (see
\cite[Theorem 2.5.10]{ADRS97}).

Consider the boundary triplet from Theorem \ref{GRes1}. The
extension $\w V_U$ is minimal. By Theorem \ref{GRes2}.3,
 the generalized resolvent that corresponds to  $\w V_U$ can be found by
  (\ref{gres1}), (\ref{gres2}).

Step 3. Assume that $\e(\cdot)\in S_{\w\k-\k}(\sN_2,\sN_1)$ and
$\e(\cdot)$ is not holomorphic at zero. Then there exists a point
$z_0\in\dD_e$ such that $\e(\cdot)$ is holomorphic at $1/z_0$.
Consider a new operator function $\e_0(1/\zeta):=\e(\frac{\zeta+\ov
z_0}{1+z_0 \zeta})$, which is holomorphic at zero. By \cite[Theorem
2.5.10]{ADRS97} there exists the unique simple unitary colligation
$\Delta_0=\{\cH^{[\p]},\sN_2,\sN_1;U_0\}$ such that the
characteristic function of this colligation coincides with
$\e_0(\zeta)$. Consider the isometric operator $V_0$ defined by
\eqref{E:LFT2} and the boundary triplet $\Pi_0$ for $V_0$
constructed in Lemma \ref{L:connect}. Then by the reasoning of Step
2 there exists a unitary extension $V_{U_0}$ of the operator $V_0$
corresponding to the unitary colligation
$\Delta_0=\{\cH^{[\p]},\sN_2,\sN_1;U_0\}$ via Theorem \ref{GRes1}.

The generalized resolvent of $V_0$ corresponding to this extension
$\w V_{U_0}$ is given by \eqref{gres1} for $\zeta\in \dD_e$.
\begin{equation}\label{gres11}
 \mathbf{R}^0_\zeta=R_\zeta(V_1^0)
 -\zeta^{-1}\g_1^0(\zeta)\e_0(1/\zeta)\sk({I-M_1^0(\zeta)\e_0(1/\zeta)})^{-1}\g_2^{0\#}(\zeta).
\end{equation}
We define a l.r.
$$
 \w V_\e=(I+z_0 \w V_{U_0})(\w V_{U_0}+\ov z_0)^{-1}.
$$
Using formula \eqref{gres11} and connections between resolvents,
$\g$-fields and Weyl-functions of $\w V_\e$ and $\w V_{U_0}$ (see
formula \eqref{RR'} and Lemma \eqref{L:connect}), one gets
\begin{equation*}
  \begin{split}
    &\mathbf{R}^0_\zeta-R_\zeta(V_1^0)=\frac{\l-z_0}{|z_0|^2-1}\sk[{I+(\l-z_0)\mathbf{R}_\l
    -\sk({I+(\l-z_0)R_\l(V_1)})}]\\
    &=-\frac{\l-z_0}{1-\ov z_0\l}\frac{\l-z_0}{\sqrt{|z_0|^2-1}}\g_1(\l)\e\sk({1/\l})
    \sk({I-M_1(\l)\e\sk({1/\l})})^{-1}
    \frac{1/\l-\ov z_0}{\sqrt{|z_0|^2-1}}\g_2^\#(\l).
  \end{split}
\end{equation*}
After simplifications we obtain the formula
\begin{equation*}
  \begin{split}
    \mathbf{R}_\l=R_\l(V_1)
    -\l^{-1}\g_1(\l)\e(1/\l)\sk({I-M_1(\l)\e(1/\l)})^{-1}\g_2^\#(\l),
  \end{split}
\end{equation*}
which proves \eqref{gres1}. Now by taking "sharp" of both sides we
can get formula \eqref{gres2} (see Remark \ref{connection} and the
last part of Theorem \ref{GRes2} proof).
\end{proof}
\begin{corollary}\label{C:cores}
  Let the assumptions of Theorem \ref{GRes2} be satisfied. Then formulas for generalized coresolvents have the following forms:
\begin{equation*}
 \begin{split}
  &P_\cH(I-\l \w V)\upharpoonright\cH=(I-\l V_1)^{-1}
  +\g_1(1/\l)\T(\l)\sk({I-M_1(1/\l)\T(\l)})^{-1}\g_2(\l)^\zx
 \end{split}
\end{equation*}
for $\l\in\r(\w V)^{-1}\cap\cD_2$;
\begin{equation*}
 \begin{split}
  &P_\cH(I-\l \w V)\upharpoonright\cH=(I-\l V_2)^{-1}
  -\g_2(1/\l)\T^\#(\l)\sk({I-M_2(1/\l)\T^\#(\l)})^{-1}\g_1(\l)^\zx
 \end{split}
\end{equation*}
for  $\l\in\r(\w V)^{-1}\cap\cD_1$.
\end{corollary}

\bigskip

CONTACT INFORMATION

\noindent
 Dmytro V. Baidiuk\\University of Vaasa\\PL/ P.O. Box 700 FI-65101 Vaasa, Finland \\dbaidiuk@uwasa.fi

\end{document}